\documentclass[11pt]{article}

\usepackage{amsmath,amsthm,amssymb,amsxtra,amsfonts,amsbsy,mathrsfs}
\usepackage[all,2cell]{xy} \UseAllTwocells
\SilentMatrices
\newtheorem{theorem}[subsection]{Theorem}

\newtheorem{lemma}[subsection]{Lemma}
\newtheorem{proposition}[subsection]{Proposition}

\theoremstyle{definition}
\newtheorem{notation-convention}[subsection]{Notations
and Conventions}

\newtheorem{remark}[subsection]{Remark}

\begin{document}
\title{Noncommutative multiplicative norm identities for the quaternions and the octonions}
\author{An Huang\footnote{anhuang@berkeley.edu}\\
Department of Mathematics\\
University of California, Berkeley\\
CA 94720-3840 USA\\}
\date{February 2011}
\maketitle
\begin{abstract}
We present Capelli type identities associated with the quaternions and the octonions, which are noncommutative versions of multiplicative norm identities for the quaternions and the octonions.
\end{abstract}
\section{Introduction}
We found the following identity in the Weyl algebra $\mathbb{Q}[x_1,x_2,x_3,x_4,\frac{\partial}{\partial x_1},\frac{\partial}{\partial x_2},\frac{\partial}{\partial x_3},\frac{\partial}{\partial x_4}]$:

Let $A=\begin{bmatrix}-x_1&x_2&x_3&x_4\\-x_2&-x_1&x_4&-x_3\\-x_3&-x_4&-x_1&x_2\\-x_4&x_3&-x_2&-x_1\end{bmatrix}$, 
and let $B$ to be the transpose of $A$, where one substitutes each $x_i$ by $\frac{\partial}{\partial x_i}$.

Then we have
\begin{equation}\label{1}
\det A\det B=\det(AB+\text{diag})
\end{equation}
where det on the right hand side means column-determinant, (one may see \cite{toru} for the definition, or wiki Capelli's identity.) and $\text{diag}=\begin{bmatrix}2&0&0&0\\0&0&0&0\\0&0&0&0\\0&0&0&-2\end{bmatrix}$,  $\begin{bmatrix}2&0&0&0\\0&0&0&0\\0&0&-2&0\\0&0&0&0\end{bmatrix}$, or $\begin{bmatrix}0&0&0&0\\0&2&0&0\\0&0&0&0\\0&0&0&-2\end{bmatrix}$.

It is easy to see that $\det(A)=(x_1^2+x_2^2+x_3^2+x_4^2)^2$. Note that if entries of $A$ and $B$ were commuting, we would not have the additional diagonal term on the right hand side of \eqref{1}, and the identity would become the square of the usual multiplicative norm identity $N(\alpha)N(\overline{\beta})=N(\alpha\overline{\beta})$ for the quaternions, where $N$ is the norm function, and $\overline{.}$ denotes conjugation in $\mathbb{H}$. 

\eqref{1} can be proved by direct calculation, although the calculation could be quite long. In the following, we describe some background of this identity and its discovery. Then we explore closer its relation with normed division algebra over $\mathbb{R}$, and state these identities in a conceptually clearer way (theorem 1.3), for which a generalization to the octonionic case is clear to guess. Section 2 is devoted to the proof of proposition 1.1.

\eqref{1} is a Capelli type identity. The subject of Capelli identity and it's generalizations has a long history. Some classical references are \cite{roger}, \cite{toru}, \cite{HW}. However, we found \eqref{1} somewhat by accident when we were mainly concerned with the problem of applications of Bernstein-Sato polynomials in perturbative quantum field theory. As it is implicitly mentioned in \cite{reb}, finding quick ways to calculate the Bernstein-Sato polynomial for invariant polynomials of the orthogonal group can be practically very useful in carrying out renormalizations. On the other hand, the classical Capelli identity provides an example for which such a goal is achieved. Namely, for the polynomials $\det(x_{ij})_{1\leq i,j\leq n}$. Therefore we were trying to generalize this idea, to see if we could find similar identities for some other invariant polynomials of the orthogonal group. Then we arrived at \eqref{1} first as a conjecture, then as a theorem. However, \eqref{1} is not of any use for our original purpose, instead, we think it is interesting in its own. 

To get an idea on $\text{diag}$, consider the usual left action of $SO(4)$ on the vector space of polynomials in $x_1,x_2,x_3,x_4$. (The restriction to degree one polynomials gives the standard representation.) Non-diagonal entries in $AB$ are differential operators which are in the image of the Lie algebra $Lie(SO(4))$. Therefore, they act as zero on any invariant polynomial of $SO(4)$. Take the homogeneous invariant polynomial $P(x_1,x_2,x_3,x_4)=(x_1^2+x_2^2+x_3^2+x_4^2)^2$. Pick any complex number $s$ with $Re(s)>0$, let us act on the distribution $P^{s+1}$ using differential operators from both sides of \eqref{1}. Since $\det{B}$ is equal to the square of the Laplace operator in $R^4$, one directly calculates that the action of the left hand side gives a scalar multiple by $(4s+6)(4s+4)^2(4s+2)$, whereas diagonal terms in $AB$ act as scalar multiplications by $(4s+4)$. From this one sees that in order for \eqref{1} to hold, one has to add a permutation of $2,0,0,-2$ to the diagonal of $AB$.

Furthermore, we will prove the following in section 2:
\begin{proposition}
If a $n\times n$ matrix $A$ has the following properties\\
(i)Every entry of $A$ is one of $x_n,x_{n-1},...,x_1,-x_1,...,-x_{n-1},-x_n$
(elements in a polynomial ring), and up to signs, every row and column of
A is a permutation of ${x_1,x_2,...,x_n}$.\\
(ii) For every $2\times 2$ submatrix $\begin{bmatrix}a&b\\c&d\end{bmatrix}$ of $A$, if $a=d$, then $b=-c$, and if
$a=-d$, then $b=c$. Vise versa.\\
(iii) Diagonal terms are all equal.

Then it comes from a multiplication table of a normed division algebra over $\mathbb{R}$
of dimension $n$. In particular, $n=1,2,4$, or $8$.

Conversely, any multiplication table of $\mathbb{H}$ or $\mathbb{O}$ gives such a matrix.
\end{proposition}
\begin{remark}
This proposition allows one to identify a multiplication table of a normed division algebra over $\mathbb{R}$ with a matrix $A$ satisfying the above purely combinatorial properties.
\end{remark} 

Having this proposition, \eqref{1} Can be formulated more clearly in the following way:
\begin{theorem}
Pick any $4\times 4$ matrix $A$ satisfying the above three properties listed in proposition 1.1. Then it provides $\mathbb{H}$ with an orthonormal basis $e_1=1, e_2, e_3, e_4$. Under this basis, we have
\begin{equation}
\det(L_{\alpha})\det(L_{\overline{\beta}})=\det(L_{\alpha\overline{\beta}}+\text{diag})
\end{equation}
where $\alpha=x_1e_1+x_2e_2+x_3e_3+x_4e_4$, $\beta=\frac{\partial}{\partial x_1}e_1+\frac{\partial}{\partial x_2}e_2+\frac{\partial}{\partial x_3}e_3+\frac{\partial}{\partial x_4}e_4$, $L_{\alpha}$ means the matrix of left multiplication by $\alpha$, under this basis. The determinant on the right hand side means column-determinant, and $\text{diag}$ equals one of $\begin{bmatrix}2&0&0&0\\0&0&0&0\\0&0&0&0\\0&0&0&-2\end{bmatrix}$, $\begin{bmatrix}2&0&0&0\\0&0&0&0\\0&0&-2&0\\0&0&0&0\end{bmatrix}$, and $\begin{bmatrix}0&0&0&0\\0&2&0&0\\0&0&0&0\\0&0&0&-2\end{bmatrix}$, if $e_2e_3e_4=1$, and equals the negative of one of these matrices, if $e_2e_3e_4=-1$.
\end{theorem}
\begin{proof}
By direct computation using Macaulay2. One may also do the verification by hand. By making careful use of commutation relations, the computation required can be reduced significantly.
\end{proof}
\begin{remark}
Here we are working in something like "quaternions over a Weyl algebra", in a clear-to-understand sense. Also note that as a corollary, if we take conjugate of $AB$ in \eqref{1}, and change $\text{diag}$ to -$\text{diag}$, \eqref{1} still holds. 
\end{remark}
Note that from the above statement of the identity, it is clear that one should guess a similar result holds for $\mathbb{O}$. (The 1 and 2 dimensional identities of this type are trivial.) $\text{diag}$ for this possible identity is easy to figure out, by the same way as we did for the $4\times 4$ case. This time, diagonal entries for $\text{diag}$ should be a permutation of ${6,4,2,0,0,-2,-4,-6}$. We are not able to verify this identity at the moment, but we have Macaulay2 code for its verification. Given sufficient internal memory and computing power, this identity can be proved or disproved very soon.
\begin{remark}
One may recognize possible relation between diagonal entries of $\text{diag}$, and degrees of fundamental invariants of the Weyl group of $SO(2n)$, as in the case of the classical Capelli identity.
\end{remark}
\section{Proof of proposition 1.1}
\begin{proof}
From properties (ii) and (iii) we know $A$ is skew symmetric. By possibly redefining basis elements, we may assume that the diagonal terms of $A$ are all equal to $-x_1$, and that the first row is $\left[-x_1,x_2,...,x_n\right]$. We define an $\mathbb{R}$ algebra $D$ which, as an $\mathbb{R}$ vector space, is generated by a basis $x_1=1,x_2,...,x_n$, and whose multiplication is given by $x_ix_j=A_{ij}, i,j\neq 1$ on basis elements, and $1$ is the multiplicative identity. We first show that $D$ is an division algebra over $R$. To this end, it suffices to show that it has no zero divisors. Suppose one has $ef=0$, where $e=\sum_{i=1}^{n}a_ix_i$, $f=-b_1x_1+\sum_{i=2}^{n}b_ix_i$ are two elements in $D$. Further assume that $e\neq 0$, we are going to show that $f=0$. We denote the column vectors $\left[a_1 a_2... a_n\right]^T$ and $\left[b_1 b_2... b_n\right]^T$ by $L$ and $R$, respectively. Thus in $D$ we have
\begin{equation}\label{F}
L^TAR=0. 
\end{equation} 
The matrix $A$ has the form $\left[-X, U_2X, U_3X,..., U_nX\right]$, where $X=\left[x_1,x_2,...,x_n\right]^T$, and $U_i$ are skew-symmetric $n\times n$ orthogonal matrices, $i=2,3,...,n$.

We have the following lemmas:
\begin{lemma}
Let $k$ be a field with characteristic not equal to $2$, $M\in k^{m\times m}$, then the following are equivalent:\\
(i) For any $P\in k^{m\times 1}$, $P^TMP=0$\\
(ii)$M^T=-M$
\end{lemma}
\begin{proof}
This is well known.
\end{proof}
\begin{lemma}
$U_iU_j=-U_jU_i$, if $i\neq j$, $2\leq i,j\leq n$.
\end{lemma}
\begin{proof}
This lemma follows from property (ii) of the matrix $A$, by diagram chasing, (i.e. apply both sides to basis elements.) somewhat subtly. But property (ii) works just in the correct way that ensures this lemma to be true.
\end{proof}
Now let us prove proposition 1.1. \eqref{F} is a statement that the coefficient of each $x_i$ is $0$. Thus it must be true if we replace $X$ by $U_2X$. Then, using lemma 2, we may write the new equation as
\begin{equation}\label{2}
L^TU_2^TAK_2R=0
\end{equation}
where $K_2$ is a diagonal matrix, with diagonal entries ${1,1,-1,...,-1}$.

Replacing $X$ by $U_iX$, we get similar equations. Now let $x_1=-1,x_2=...=x_n=0$, we therefore have $M_LR=0$, where $M_L$ is the matrix consisting of rows $L^T, L^TU_2^TK_2,...,L^TU_n^TK_n$. We shall see that under the usual inner product of $R^n$, these rows are orthogonal to each other, thus $M_L$ is invertible, and $R=0, f=0$. It is easy to see that the first row is orthogonal to every other row. Without loss of generality, it suffices for us to show that the second row is orthogonal to the third row. According to lemma 1, it suffices to show that the matrix $U_2^TK_2K_3U_3$ is skew-symmetric. We denote $S=K_2K_3$. Relabeling $x_i$ and changing sign if necessary, we may assume $U_2x_3=-x_4$. Then by property (ii) of $A$, we have $U_2x_4=-x_3$. Furthermore, $U_3x_4=-U_3U_2x_3=U_2U_3x_3=-U_2x_1=-x_2$, $U_3x_2=-x_4$. From the matrix identity
\begin{equation}
\begin{bmatrix}0&1&0&0\\1&0&0&0\\0&0&0&1\\0&0&1&0\end{bmatrix}\begin{bmatrix}0&0&-1&0\\0&0&0&-1\\1&0&0&0\\0&1&0&0\end{bmatrix}=
\begin{bmatrix}0&0&1&0\\0&0&0&-1\\1&0&0&0\\0&-1&0&0\end{bmatrix}\begin{bmatrix}0&-1&0&0\\1&0&0&0\\0&0&0&1\\0&0&-1&0\end{bmatrix}
\end{equation}
we see that
\begin{equation}
[U_2,S]U_3+[U_3,S]U_2=0
\end{equation}
Which implies that $U_2^TK_2K_3U_3$ is skew-symmetric, by an easy calculation.

Furthermore, properties (i), (ii) show that the multiplication in $D$ preserves the Euclidean norm. (one can convince oneself by using the definition of norm, and property (ii) ensures that all "crossing terms" cancel.) Thus $D$ is a normed division algebra over $\mathbb{R}$.

Lastly, one can verify directly that any multiplication table of $\mathbb{H}$ or $\mathbb{O}$ gives rise to a matrix $A$ with the above two properties. One even has a uniform way of seeing this, by an use of the Moufang identity, as follows:

Suppose we have such a multiplication table. Then property (i) is obviously satisfied. Let us assume that $j_1i_1=j_2i_2$, we are going to show $j_1i_2=-j_2i_1$. Where $j_1, j_2$ are basis elements in the column, and $i_1, i_2$ are basis elements in the row. (Conversely, exactly the same method applies)

We have $(j_1i_1)(j_1i_2+j_2i_1)=(j_2i_2)(j_1i_2)+(j_1i_1)(j_2i_1)=-i_2(j_2j_1)i_2-i_1(j_1j_2)i_1=0$, thus $j_1i_2+j_2i_1=0$ as we are working in a division algebra. In the above, we have used the fact that $\mathbb{H}$ and $\mathbb{O}$ are alternative algebras, for which Moufang identities apply.  
\end{proof}
\section{Acknowledgment}
I thank Richard Borcherds for suggesting the problem on Bernstein-Sato polynomial, thank Toru Umeda for inspiring discussions, and thank Claudiu Raicu for helping me with the Macaulay2 code. I am also grateful to Zhiqi Chen, Roger Howe, and Shaowei Lin for useful correspondences.

\end{document}